\documentclass[a4paper, oneside, 10pt]{article} \usepackage[english]{babel}
\usepackage{latexsym,amsfonts,amsmath,enumerate,graphics,enumerate,amsthm}
\title{\textbf{A Note on Entropy of Delone Sets} } \author{T.~Hauser}

\let\epsilon=\varepsilon
\makeatletter
\newtheorem*{rep@theorem}{\rep@title}
\newcommand{\newreptheorem}[2]{%
\newenvironment{rep#1}[1]{%
 \def\rep@title{#2 \ref{##1}}%
 \begin{rep@theorem}}%
 {\end{rep@theorem}}}
\makeatother

\newreptheorem{theorem}{Theorem}
\newreptheorem{corollary}{Corollary}
\newreptheorem{proposition}{Proposition}

\theoremstyle{definition}
\newtheorem{definition}{Definition}[section]

\newtheorem{theorem}[definition]{Theorem}
\newtheorem{proposition}[definition]{Proposition}
\newtheorem{lemma}[definition]{Lemma}
\newtheorem{corollary}[definition]{Corollary}

\newtheorem*{acknowledgement}{Acknowledgement}

\theoremstyle{remark}
\newtheorem{remark}[definition]{Remark}
\newtheorem{example}[definition]{Example}

\let\epsilon=\varepsilon
\let\phi=\varphi
\let\theta=\vartheta

\begin{document}
\maketitle

%

\begin{abstract}
	In this note we present that the patch counting entropy can be obtained as a limit and investigate which sequences of compact sets are suitable to define this quantity. We furthermore present a geometric definition of patch counting entropy for Delone sets of infinite local complexity and that the patch counting entropy of a Delone set equals the topological entropy of the corresponding Delone dynamical system. 
	We  present our results in the context of (non-compact) locally compact abelian groups that contain Meyer sets. 
\end{abstract}



\section{Introduction}

	The study of structural properties of Delone sets is a key topic in the context of quasicrystals and aperiodic order. The concept of a Delone set can be seen as a mathematical abstraction for the positions of atoms in a solid state material, which in addition yields a physical motivation for this investigation. To define this concept we denote  $A+B:=\{a+b;\, a\in A, b\in B\}$ for $A,B \subseteq \mathbb{R}^d$ and define similarly $A-B$ and $A+g$ for $g\in \mathbb{R}^d$.
	A subset $\omega\subseteq \mathbb{R}^d$ is called \emph{uniformly discrete}, whenever there exists an open neighbourhood of $0$ such that $\{V+g;\, g\in \omega\}$ is a disjoint family and \emph{relatively dense}, whenever there is a compact set $K\subseteq \mathbb{R}^d$ such that $K+\omega=\mathbb{R}^d$. A subset $\omega\subseteq G$ is called \emph{Delone}, whenever it is relatively dense an uniformly discrete. 
	
An important tool in the study of Delone sets is the \emph{patch counting entropy} of a Delone set $\omega$. This notion was studied in \cite{baake2007pure} and can be found in \cite{lagarias1999geometric,lagarias2003repetitive} 
	under the name of \emph{configurational entropy}. 		
	For a Delone set $\omega$ and a compact set $A\subseteq \mathbb{R}^d$ we define the set of all \emph{$A$-patches} as $\operatorname{Pat}_\omega(A):=\{(\omega-g)\cap A;\, g\in \omega\}$, and $\omega$ is said to be of \emph{finite local complexity}, whenever $\operatorname{Pat}_\omega(A)$ is finite for all compact subsets $A\subseteq \mathbb{R}^d$. For Delone sets of finite local complexity we define the \emph{patch counting entropy} as
	\begin{align}\label{form:EPAT}
		\operatorname{E}_{pc}(\omega):=\limsup_{n \to \infty}\frac{\log(|\operatorname{Pat}_\omega(B_n)|)}{\theta(B_n)},	
	\end{align}
	where $B_n$ denotes the centred and closed Euclidean ball with radius $n$, $|\cdot|$ the cardinality and $\theta$ is the Lebesgue measure. 
	In \cite{lagarias2003repetitive} it is claimed that  in (\ref{form:EPAT}) one has always a limit and that 
''existence of the limit can be established by a subadditivity argument.'' Following this idea we will use a tool called the \emph{Ornstein-Weiss lemma} to obtain the existence of the limit. This lemma considers functions $f$ mapping compact subsets of $\mathbb{R}^d$ to positive real numbers such that for all compact sets $A,B\subseteq \mathbb{R}^d$ we have that $f(A\cup B)\leq f(A)+f(B)$ (sub-additivity); $f(A)\leq f(B)$ whenever $A\subseteq B$ (monotonicity); and $f(A+g)=f(A)$ for all $g\in G$ (invariance). The lemma then states that the limit $\lim_{n\to \infty}{f(B_n)}/{\theta(B_n)}$ exists.   
  	Unfortunately the function $A\mapsto |\operatorname{Pat}_\omega(A)|$ is not invariant. Consider for example the Delone set $\omega:=\mathbb{Z}\cup (\mathbb{Z}+\epsilon)$ in $\mathbb{R}$ for some $0<\epsilon<1/4$. Then for $A:=[\epsilon,2\epsilon]$ we have that $\operatorname{Pat}_\omega(A)=\{\{\epsilon\},\emptyset\}$, but $\operatorname{Pat}_\omega(A+\epsilon)=\{\emptyset\}$. Thus a direct application of the Ornstein-Weiss lemma is not possible. We will use the Ornstein-Weiss lemma in combination with the theory of topological entropy in order to obtain the existence of the limit. 
  	One therefore considers for each Delone set $\omega$ the closure $X_\omega$ of the set $\{\omega-g;\, g\in \mathbb{R}^d\}$ of all translates of $\omega$ with respect to a suitable topology on $\mathcal{A}(\mathbb{R}^d)$, the space of all closed subsets of $\mathbb{R}^d$. Then the shift $\mathbb{R}^d\times X_\omega \ni (g,\xi)\mapsto\xi+g$ introduces a dynamical system $\pi_\omega$, called the \emph{Delone dynamical system}. For details on this construction, see Subsection \ref{sec:Prelims} below. In \cite{baake2007pure} it is shown that the topological entropy of $\pi_\omega$ equals the patch counting entropy $\operatorname{E}_{pc}(\omega)$, whenever $\omega$ is of finite local complexity. A careful analysis of the underlying arguments in combination with the Ornstein-Weiss lemma indeed gives, that the limit in (\ref{form:EPAT}) always exists. We will present the arguments in Section \ref{sec:FLC} below. 
	
	In \cite{HuckandRichard} the question was raised, which type of sequences of compact sets can be considered in (\ref{form:EPAT}) instead of $(B_n)_{n\in \mathbb{N}}$ and Patch counting entropy is studied along Van Hove sequences. A \emph{Van Hove sequence} is a sequence $(A_n)_{n\in \mathbb{N}}$ of compact subsets such that for all compact $K\subseteq G$ we have that $\lim_{i\in I}\theta(\partial_K A_i)/\theta(A_i)=0$, where we define 
	$\partial_K A:=(K+\overline{A})\cap (K+\overline{A^c})$ for all compact $A\subseteq \mathbb{R}^d$.\footnote{Here $\overline{A}$ denotes the closure and $A^c$ the complement within $\mathbb{R}^d$.}
	Note that the Ornstein-Weiss lemma above holds for all Van Hove sequences $(A_n)_{n\in \mathbb{N}}$ and that the limit $\lim_{n\to \infty}{f(A_n)}/{\theta(A_n)}$ is independent of the choice of such a sequence for all monotone, invariant and sub-additive functions $f$. Nevertheless (\ref{form:EPAT}) is not independent of the choice of a Van Hove sequence.  
	Indeed, the following example shows that there are Delone sets such that one can obtain each value in $[0,\infty]$ in (\ref{form:EPAT}) for a suitable choice of a Van Hove sequence. 
\begin{example}\label{exa:introcompactlyconnectednecessary}
	Consider the finite local complexity Delone set $\omega:=(-\mathbb{N}_0)\cup \alpha\mathbb{N}_0\subseteq \mathbb{R}$ for $\alpha\in [0,1]$ irrational. Then for $\kappa\in [0,\infty]$ one obtains
	\begin{align}\label{form:fromexample}
		\limsup_{n \to \infty}\frac{\log(|\operatorname{Pat}_\omega(A_n)|)}{\theta(A_n)}=\kappa,
	\end{align}
	whenever we choose $A_n:=[0,n]+e^{\kappa n}$ if $\kappa$ is finite and $A_n:=[0,n]+e^{(n^2)}$ if $\kappa=\infty$. For details see Example \ref{exa:compactlyconnectednecessary}. 
\end{example}

Note that the effect considered in Example \ref{exa:introcompactlyconnectednecessary} can also be obtained from Van Hove nets with sub-exponential distance from $\{0\}$ and in fact for every sequence $(b_n)_{n\in \mathbb{N}}$ in $[0,\infty)$ with $b_n\rightarrow  \infty$ we can set $A_n:=[0,\log(b_n)/\kappa]+ b_n$ to obtain (\ref{form:fromexample}) for $\kappa\in (0,\infty)$. 
We will thus consider Van Hove nets that ''stay close to $0$'' and use the following notion inspired from \cite{OrnsteinandWeiss}.	
	Let $C\subseteq G$ be a compact subset. We say that $A\subseteq G$ is $C$-connected to $0$, if for all $a\in A$ there are $a_0,\cdots,a_n\in A\cup \{0\}$  with $a_0=0$, $a_n=a$ and $a_i-a_{i-1}\in C$ for every $i\in \{1,\cdots,n\}$. Furthermore we say that a sequence of compact sets $(A_n)_{n\in \mathbb{N}}$ is \emph{$C$-connected to $0$}, if $A_n$ is $C$-connected to $0$ for all $n\in \mathbb{N}$. A net is called \emph{compactly connected to $0$}, if it is $C$-connected to $0$ for some compact set $C\subseteq G$. Note that every sequence of compact path-wise connected sets that contains $0$, such as $(B_n)_{n\in \mathbb{N}}$, is compactly connected to $0$.   
	
\begin{theorem}\label{the:EPATintro}
	Let $\omega$ be a Delone set of finite local complexity and $(A_n)_{n\in \mathbb{N}}$ a Van Hove sequence that is compactly connected to $0$. Then the following limit exists and 
		\begin{align*}
		\operatorname{E}_{pc}(\omega)=\lim_{n \to \infty}\frac{\log(\operatorname{Pat}_\omega(A_n))}{\theta(A_n)}.
	\end{align*}
\end{theorem}	 

	Note that the definition of topological entropy is independent of the choice of a larger class of sequences, so called F\o lner  sequences. Thus a natural next question is, whether the previous theorem holds true, whenever we consider F\o lner sequences and {ergodic} sequences, which are compactly connected to $0$. 
	These notions are studied in \cite{Tempelman,Pier} and defined as follows. Consider a sequence of compact sets $(A_n)_{n\in \mathbb{N}}$. We call $(A_n)_{n\in \mathbb{N}}$  \emph{ergodic}, whenever for all $g\in G$ we have that $\lim_{n\to \infty}\theta((A_n+g)\Delta A_n)/\theta(A_n)=0$ and \emph{F\o lner}, whenever for all $K\subseteq G$ compact we have $\lim_{n \to \infty}\theta((A_n+K)\Delta A_n)/\theta(A_n)=0$. Here $\Delta$ denotes the symmetric difference.\footnote{The \emph{symmetric difference} of sets $A,B$ is defined as $A\Delta B:=(A\setminus B)\cup (B\setminus A)$.} 
	Note that every Van Hove sequence is F\o lner and that every F\o lner sequence is ergodic. Furthermore note that these notions are equivalent in the context of discrete groups.  Nevertheless these notions are not equivalent in $\mathbb{R}^d$ and the next examples show that one can not consider F\o lner sequences in Theorem \ref{the:EPATintro}.
	
	\begin{example}\label{exa:introFolnerrightergodic}
		Consider the Delone set of finite local complexity
			\[\omega:=\{n\in \mathbb{N};\, \xi_n=1\}\cup (\mathbb{Z}+1/2),\] where $(\xi_n)_{n\in \mathbb{N}}$ is a sequence containing all finite words in $\{0,1\}$, i.e.\ for all finite sequences $(x_j)_{j=1}^n$ there exists $i\in \mathbb{N}$ such that $\xi_{i+j}=x_j$ for $j=1,\cdots,n$. 		
	Then $\operatorname{E}_{pc}(\omega)=\log(2)$ and for all $\kappa\in[0,\operatorname{E}_{pc}(\omega)]$ there is a F\o lner sequence $(A_n)_{n\in \mathbb{N}}$, which is compactly connected to $0$, such that
		\begin{align}\label{form:EPATforFolner}
		\limsup_{n \to \infty}\frac{\log(\operatorname{Pat}_\omega(A_n))}{\theta(A_n)}=\kappa.
	\end{align}
	Furthermore for all $\kappa\in[0,\infty]$ there is an ergodic sequence $(A_n)_{n\in \mathbb{N}}$, which is compactly connected to $0$, such that 
		\begin{align*}
		\limsup_{n \to \infty}\frac{\log(\operatorname{Pat}_\omega(A_n))}{\theta(A_n)}=\kappa.
	\end{align*}
See Example \ref{exa:Folnerfails} and Example \ref{exa:ergodicfails} in Section \ref{sec:examples} for details.  
	\end{example}  
	
	In fact the boundedness of the limit superior by $\operatorname{E}_{pc}(\omega)$ in (\ref{form:EPATforFolner}) above is a general phenomenon. 
	
	\begin{theorem}\label{the:introFolnerleqEPat}
	If $(A_n)_{n\in \mathbb{N}}$ is a F\o lner sequence that is compactly connected to $0$, then 
	 \begin{align*}
		\limsup_{n \to \infty}\frac{\log(\operatorname{Pat}_\omega(A_n))}{\theta(A_n)}\leq \operatorname{E}_{pc}(\omega).
	\end{align*}
	\end{theorem}

	In \cite{frank2012fusion,frank2015tilings,fuhrmann2018irregular} patch counting entropy for arbitrary Delone sets is considered. Note that (\ref{form:EPAT}) always gives the value $\infty$ for Delone sets, which are not of finite local complexity. We refer to such Delone sets as sets of \emph{infinite local complexity}. The idea stated in \cite{frank2012fusion,frank2015tilings} and also used in \cite{fuhrmann2018irregular} is to take the topological entropy of the corresponding Delone dynamical system as the definition of patch counting entropy for all Delone sets of infinite local complexity. The problem with this approach is that one has to compute and understand the corresponding  Delone dynamical system if one wishes to compute the patch counting entropy. We will thus not follow this road completely and instead define the following notion of ''patch counting for Delone sets if infinite local complexity''. 	
	For subsets $\xi,\zeta,A,V\subseteq \mathbb{R}^d$, with $A$ compact and $V$ an open neighbourhood of the identity element, we write
	\[\xi\overset{A,V}{\approx}\zeta,\] whenever $\xi$ and $\zeta$ agree in $A$ up to an error of $V$, i.e.\ if $\xi\cap A\subseteq \zeta+V$ and $\zeta\cap A\subseteq \xi+V$.   
	For a Delone set $\omega\subseteq \mathbb{R}^d$ we say that $F\subseteq \omega$ is an \emph{$A$-patch representation} of \emph{scale} $V$ for $\omega$, if for any $g\in \omega$ there is $f\in F$ s.t. \[\omega-f \overset{A,V}{\approx} \omega-g.\]
In Remark \ref{rem:repandspafinite} we will present that there is always a finite $A$-patch representation of scale $V$ for $\omega$. We define $\operatorname{pat}_\omega(A,V)$ 
	as the minimal cardinality of an $A$-patch representation of scale $V$ for $\omega$. We will present in Section \ref{sec:PatchcountingGeneral} that this approach yields the patch counting entropy as defined in (\ref{form:EPAT}). 
	
	\begin{theorem}\label{the:introEpatandEPAT}
		Let $\omega$ be a Delone set of finite local complexity. Then 
\begin{align*}
		\operatorname{E}_{pc}(\omega)=\sup_{V\in \mathcal{N}(G)}\limsup_{n\to \infty} \frac{\log(\operatorname{pat}_\omega (B_n,V))}{\theta(B_n)},	
	\end{align*}
	\end{theorem}
	
	The intriguing thing about the approach via $\operatorname{pat}_\omega$ is that this approach is independent of the choice of an arbitrary Van Hove sequence and the assumption of connectedness to $0$ is not necessary in this context. 
	
	\begin{theorem}\label{the:intropatapproachVanHoveindependence}
		Let $\omega$ be a Delone set. Then for every Van Hove sequence $(A_n)_{n\in \mathbb{N}}$ we have  
	\begin{align}\label{form:Epat}
		\operatorname{E}_{pc}(\omega)=\sup_{V\in \mathcal{N}(G)}\limsup_{n\to \infty} \frac{\log(\operatorname{pat}_\omega (A_n,V))}{\theta(A_n)},	
	\end{align}
	\end{theorem}
%
	
	Note that the Delone set $\omega$ considered in Example \ref{exa:introFolnerrightergodic} above is contained in $\frac{1}{2}\mathbb{Z}$ and thus for any open neighbourhood $V$ of $0$ that is contained in the open centred ball of radius $\frac{1}{4}$ we have that $\operatorname{pat}_\omega(A,V)=|\operatorname{Pat}_\omega(A)|$ for any compact set $A$. One thus obtains that there are ergodic sequences and F\o lner sequences $(A_n)_{n\in \mathbb{N}}$ such that the supremum in (\ref{form:Epat}) attains exactly the values in $[0,\infty]$ and $[0,\operatorname{E}_{pc}(\omega)]$ respectively.

	We will follow investigations like \cite{HuckandRichard}, \cite{Schlottmann} or\cite{baake2004dynamical}, where Delone sets are studied in more general groups, such as (non-compact) locally compact abelian groups with certain countability assumptions. To avoid the unnecessary countability assumptions we use the more general notion of nets from now on and recommend to think of sequences if the reader is not familiar with this notion. Unfortunately we do not know of a reference of the Ornstein-Weiss lemma for general locally compact abelian groups. Nevertheless Delone sets are naturally studied in locally compact abelian groups that contain Meyer sets, i.e.\ Delone sets such that there is a finite subset $F\subseteq G$ such that $\omega-\omega\subseteq \omega+F$. We will not introduce the notion of a cut and project scheme \cite{Schlottmann},\cite{baake2004dynamical}, but note that the existence of a cut and project scheme with $G$ as a ''physical space'' is equivalent to the existence of a Meyer set in $G$ and ensures the Ornstein-Weiss lemma to hold. For reference see \cite{meyer1972algebraic,04}. We will see that all the results mentioned above hold true in this setting and in particular the topological entropy of a Delone set $\pi_\omega$ equals the patch counting entropy, which generalizes a result from \cite{baake2007pure}. 
	
	The article is organized as follows. We first give some preliminaries in Section \ref{sec:Prelims}. In Section \ref{sec:PatchcountingGeneral} we will investigate the approach of patch counting entropy via $\operatorname{pat}_\omega$ and the topological entropy of the Delone dynamical system. In Section \ref{sec:FLC} we will then restrict to Delone sets of finite local complexity and establish the results about $\operatorname{Pat}_\omega$. In Section \ref{sec:examples} we present the details on the mentioned examples.

\section{Preliminaries}\label{sec:Prelims}

\subsection{Locally Compact Abelian groups}

Consider an abelian group $G$. We write $0$ for the neutral element in $G$ and use the additive notion. For subsets $A,B\subseteq G$ the \emph{Minkowski sum} is defined as 
$A+B:=\{a+b;\, (a,b)\in A \times B\}.$ Similarly one defines the notions $A+g$, $g+A$, $-A$ for $g\in G$ and $A\subseteq G$. We denote the \emph{complement} $A^c:=G\setminus A$ and call $A$ \emph{symmetric}, if $A=-A$. Note that the complement and the inverse commute, i.e. $-(A^c)=(-A)^c$. In order to omit  brackets, we will use the convention, that the inverse and the complement are stronger binding than the Minkowski sum, which is stronger binding than the remaining set theoretic operations. 

A \emph{locally compact abelian group} (\emph{LCA group}) is an abelian group $G$ equipped with a locally compact Hausdorff topology $\tau$, such that the the addition $+\colon G\times G \to G$ and the inverse $-(\cdot)\colon G\to G$ are continuous. An \emph{isomorphism of LCA groups} is a homeomorphism that is a group homomorphism as well. We write $\overline{A}$ for the {closure} of a subset $A\subseteq G$. Furthermore we denote $\mathcal{K}(G)$ for the set of all non-empty compact subsets of $G$; $\mathcal{A}(G)$ for the set of all closed subsets and $\mathcal{N}(G)$ for the set of all open neighbourhoods of $0$.

A \emph{Haar measure} on $G$ is a non zero regular Borel measure $\theta$ on $G$, which satisfies $\theta(A)=\theta(A+g)$  for all $g\in G$ and all Borel sets $A\subseteq G$. Note that on all LCA groups there exists a Haar measure \cite[Proposition 10.4]{Folland} and that $\theta(U)>0$ for all non empty open $U\subseteq G$. 
Furthermore for $A\subseteq G$ we have that $\theta(A)<\infty$ whenever $A$ is pre-compact. 
The Haar measure is unique up to scaling, i.e.\ if $\theta$ and $\nu$ are Haar measures on $G$, then there is $c>0$ such that $\theta(A)=c\nu(A)$ for all Borel sets $A\subseteq G$. 
 If nothing else is mentioned, we denote a Haar measure of a topological group $G$ by $\theta$.
For further reference see \cite{Folland}.

\subsection{Uniformities}

	Let $X$ be a set. 
	A \emph{binary relation on $X$} is a subset of $X\times X$. For binary relations $\eta$ and $\kappa$ on $X$ we denote the \emph{inverse}  
$\eta^{-1}:=\{(y,x);\, (x,y)\in \eta \}$, the \emph{composition} $\eta \kappa:=\{(x,y);\, \exists z\in X : (x,z)\in \eta \text{ and } (z,y)\in \kappa\}$ and $\eta[x]:=\{y\in X;\, (y,x)\in \eta\}$.
A binary relation is called \emph{symmetric}, if $\eta=\eta^{-1}$. 



	For a compact Hausdorff space $X$, we denote the \emph{diagonal} $\Delta_X:=\{(x,x);\, x\in X\}$ and call a neighbourhood of $\Delta_X$ in $X^2$ an \emph{entourage (of $X$)}. The set of all entourages of $X$ is referred to as the \emph{uniformity\footnote{
Note that one can define general ''uniform spaces'', but as we are only interested in compact Hausdorff spaces, this definition works for us. For details and the general definition we recommend \cite{Kelley}. Note that we obtain our definition to be a restriction of the general definition from \cite[Theorem 6.22]{Kelley} and \cite[Theorem 32.3]{Munkres}.	
	} of $X$} and usually denoted by $\mathbb{U}_X$. 
To obtain some geometric intuition for $\eta\in \mathbb{U}_X$ we say that \emph{$x$ is $\eta$-close to $y$}, whenever $(x,y)\in \eta$. 
 We think of two elements to be ''very close'', whenever the pair is $\eta$-close for ''many'' entourages $\eta$.  
Note that if $x$ is $\eta$-close to $y$ and $y$ is $\kappa$-close to $z$, then $x$ is $\eta \kappa$-close to $z$.

A subfamily $\mathbb{B}_X\subseteq \mathbb{U}_X$ is called a \emph{base for $\mathbb{U}_X$}, if every entourage contains a member of $\mathbb{B}_X$. An entourage $\eta \in \mathbb{U}_X$ is called \emph{open} (or \emph{closed}), whenever it is open (or closed) as a subset of $X\times X$.  Note that the family of all open and symmetric entourages of $X$ forms a base of the uniformity of $X$. 
	If $(X,d)$ is a metric space we denote $[d<\epsilon]:=\{(x,y)\in X \times X;\, d(x,y)<\epsilon \}$ for $\epsilon<0$.  
Then $\mathbb{B}_d:=\{[d<\epsilon];\, \epsilon>0\}$ is a base for the uniformity of the corresponding  topological space $X$. 
Note that $x$ is $[d<\epsilon]$-close to $y$, if and only if $d(x,y)<\epsilon$.

\subsection{Dynamical systems}

Let $G$ be an LCA group and $X$ be a compact Hausdorff space. A continuous map $\phi\colon G\times X \to X$ is called an \emph{action of $G$ on $X$} (also \emph{dynamical system} or \emph{flow}), whenever $\phi(0,\cdot)$ is the  identity on $X$ and for all $g,g'\in G$ we have that $\phi(g,\phi(g',\cdot))=\phi(g+g',\cdot)$. We write $\phi^g:=\phi(g,\cdot)\colon X\to X$ for all $g\in G$. 


\subsection{Delone and Meyer sets}

	Let $G$ be an LCA group and let $M\subseteq G$ be a subset. A subset $\omega\subseteq G$ is called \emph{$M$-dense} in $G$, if $M+\omega=G$. Also $M$ is said to be \emph{$M$-discrete} if $\{M+g;\, g\in \omega\}$ is a disjoint family. Furthermore $\omega$ is called \emph{relatively dense}, if $\omega$ is $K$-dense for some compact $K\subseteq G$; and \emph{uniformly discrete}, if it is $V$-discrete for some open neighbourhood $V$ of $0$. If $\omega$ is relatively dense and uniformly discrete, we call $\omega$ a \emph{Delone set}. 
Note that $G$ is assumed to be Hausdorff. Thus every uniformly-discrete subset of $G$ is discrete. 
A \emph{Meyer set} is a Delone set $\omega\subseteq G$ such that there is a finite set $F\subseteq G$ that satisfies $\omega-\omega\subseteq \omega+F$. Note that all discrete LCA groups are Meyer sets in themselves.
	Furthermore it is well known that all compactly generated LCA groups are isomorphic to $\mathbb{R}^d\times \mathbb{Z}^n\times C$ for some integers $d,n$ and some compact LCA group $C$, which contains the Meyer set $\mathbb{Z}^{(d+n)}\times \{0\}$. For reference see for example \cite[Theorem 4.2.2]{deitmar2014principles}.  
	
	For a more exotic example consider the metrizable and $\sigma$-compact LCA group $\mathbb{Q}_p$ of $p$-adic integers. For a reference on $\mathbb{Q}_p$ see \cite{gouvea1997p}. $\mathbb{Q}_p$ contains no discrete subgroup other then $\{0\}$. Denote by $\mathbb{Z}[p^{-1}]$ the smallest subring of $\mathbb{Q}_p$ that contains $\mathbb{Z}$ and $p^{-1}$. Note that $\mathbb{Z}[p^{-1}]$ is also a subring of $\mathbb{R}$ and we can define $\Lambda:=[0,1]\cap \mathbb{Z}[p^{-1}]$. In fact one can show using the ''cut and project scheme'' $(\mathbb{Q}_p,\mathbb{R},\{(z,z);\, z\in \mathbb{Z}[p^{-1}]\})$ that $\Lambda$ is a Meyer set in $\mathbb{Q}_p$. For details see \cite{meyer1972algebraic,cornulier2014metric,04}.
	
	Note furthermore that there are metrizable and seperable LCA groups that contain no Meyer sets as presented in \cite[Chapter II.11]{meyer1972algebraic}.

\subsection{Delone dynamical systems}
	
	Consider a Delone set $\omega\subseteq G$. For $A\subseteq G$ compact and $g\in \omega$, we call $(\omega -g)\cap A$ an \emph{$A$-patch} of $\omega\subseteq G$ and denote the set of all $A$-patches by $\operatorname{Pat}_{\omega}(A)$. A Delone set is said to have \emph{finite local complexity} (FLC), if $\operatorname{Pat}_\omega(A)$ is finite for every compact set $A\subseteq G$. Otherwise it is said to be of \emph{infinite local complexity} (ILC).
For $K\subseteq G$ compact, an open neighbourhood $V$ of $0$ and $\xi,\zeta\in \mathcal{A}(G)$ we denote 
	\[\xi \overset{K,V}{\approx} \zeta,\]
 whenever there is $\xi\cap K\subseteq \zeta + V$ and $\zeta\cap K\subseteq \xi + V$. 
	Considering the family of all subsets $O\subseteq \mathcal{U}_G$ such that for all $\xi\in \mathcal{A}(G)$ there are a compact  subset $K\subseteq G$ and an open neighbourhood $V$ of $0$ such that $\{\zeta\in \mathcal{A}(G);\, \xi \overset{K,V}{\approx} \zeta\}\subseteq O$ we obtain a compact Hausdorff topology on $\mathcal{A}(G)$, called the \emph{local rubber topology} \cite[Theorem 3]{baake2004dynamical}.  
 Defining 
	\[\epsilon(K,V):=\left\{(\xi,\zeta)\in \mathcal{A}(G)^2;\, \xi \overset{K,V}{\approx} \zeta \right\}.\]
	 for compact $K\subseteq G$ and open neighbourhoods $V$ of $0$ we obtain that 
	 \[\mathbb{B}_{lr}:=\{\epsilon(K,V);\, (K,V)\in \mathcal{K}(G)\times \mathcal{N}(G)\}\]
	is a base for the corresponding uniformity $\mathbb{U}_{\mathcal{A}(G)}$ on $\mathcal{A}(G)$. We call this base the \emph{local rubber base}. The uniformity is called the \emph{local rubber uniformity}. For a Delone set $\omega\subseteq G$ we denote 
	$D_\omega:=\{\omega+g;\, g\in G\}$
and $X_\omega$ for the closure of $D_\omega$ with respect to the local rubber topology. 
Then $X_\omega$ is a compact Hausdorff space and 
\[\mathbb{B}_{lr}(\omega):=\{\epsilon_\omega(K,V);\, (K,V)\in \mathcal{K}(G)\times \mathcal{N}(G)\}\]
is a base of the corresponding uniformity, where we denote $\epsilon_\omega(K,V):=\epsilon(K,V)\cap (X_\omega\times X_\omega)$ for the restricted entourages. We call $\mathbb{B}_{lr}(\omega)$ the \emph{(restricted) locally rubber base} and define the \emph{Delone  dynamical system} 
\[\pi_\omega\colon G\times X_\omega \to X_\omega\] 
by $\pi_\omega(g,\xi):=\xi+g$. The continuity of this action is shown in \cite{baake2004dynamical}. 
If $\omega\subseteq G$ is a FLC Delone set there is another base of $\mathbb{U}_{X_\omega}$ that allows more control over the considered sets. For $K\subseteq G$ compact and any open neighbourhood $V$ of $0$ let 
\[\eta_\omega(K,V):=\{(\xi,\zeta)\in X_\omega^2;\, \exists x,z\in V\colon (\xi+x)\cap K=(\zeta+z)\cap K\}.\]
If $\omega\subseteq G$ is a FLC Delone set, then
	\[\mathbb{B}_{\text{lm}}(\omega):=\{\eta_\omega(K,V);\, (K,V)\in \mathcal{K}(G)\times \mathcal{N}(G)\}\]
is a base of $\mathbb{U}_{X_\omega}$. This easily follows from \cite[Prop. 4.5]{baake2004dynamical}. We will refer to this base as the \emph{local matching base} of $\mathbb{U}_{X_\omega}$.

\subsection{Amenable groups and Van Hove nets}

Let $G$ be an LCA group. 
	For $K,A\subseteq G$ we define the \emph{symmetric difference} as $A\Delta K:=(A\setminus K)\cup (K\setminus A)$ and the \emph{$K$-boundary of $A$} as
$\partial_K A:=(K+\overline{A})\cap (K +\overline{A^c})$.
	Note that $K+ \overline{A}$ is the set of $g\in G$ such that $(-K+g)\cap \overline{A}$ is not empty. 
Thus $\partial_K A$ is the set of all elements $g\in G$ such that $-K+g$ intersects both $\overline{A}$ and $\overline{A^c}$. 

For details on nets and convergence, we refer to \cite{04,DunfordandSchwartz,Kelley}. A net $(A_i)_{i\in I}$ of compact subsets of $G$ is called \emph{ergodic}, whenever we have  $\lim_{i\in I}\theta((A_i+g)\Delta A_i)/\theta(A_i)=0$  for all $g\in G$ ; \emph{F\o lner}, whenever we have $\lim_{i\in I}\theta((A_i+K)\Delta A_i)/\theta(A_i)=0$ for all non empty and compact $K\subseteq G$; and \emph{Van Hove}, whenever we have 
$\lim_{i\in I}\theta(\partial_K A_i)/{\theta(A_i)}=0$ for all compact $K\subseteq G$. Note that we implicitly assume $\theta(A_i)>0$ for large $i\in I$ in these definitions. 

Note that our definition of Van Hove nets is equivalent to the definitions given in \cite{Schlottmann} and \cite{Tempelman} and that every LCA group contains a Van Hove net \cite{04}. Furthermore every Van Hove net is F\o lner and every F\o lner net  is ergodic. Furthermore in \cite[Appendix (3.K)]{Tempelman} it is shown that an ergodic net $(A_i)_{i\in I}$ is Van Hove, whenever there exists a compact neighbourhood $U$ of $0$ such that $\lim_{i\in I}\theta(\partial_U A_i)\theta(A_i)=0$. Thus these notions are equivalent in the context of discrete groups. As seen in the Introduction these notions are pairwise non-equivalent in $\mathbb{R}^d$. 

	\begin{proposition}\label{pro:VanHovenets+K}
		Let $(A_i)_{i\in I}$ be a F\o lner net in an LCA group $G$ and $M\subseteq G$ be a compact neighbourhood of $0$.
	Then $(M+A_i)_{i\in I}$ is a Van Hove net and satisfies $\lim_{i\in I}\theta(M+A_i)/\theta(A_i)=1$.
	\end{proposition}
	\begin{proof}
	From $A_i\cup ((M+A_i)\Delta A_i)=A_i\cup ((M+A_i)\setminus A_i)\cup (A_i\setminus(M+A_i))=A_i\cup ((M+A_i)\setminus A_i)\supseteq M+A_i$ for all $i\in I$ we obtain
	\[1\leq \frac{\theta(M+A_i)}{\theta(A_i)}\leq 1+\frac{\theta((M+A_i)\Delta A_i)}{\theta(A_i)}\rightarrow 1,\]
	which shows $\lim_{i\in I}{\theta(M+A_i)}/{\theta(A_i)}=1$.
	
	To show that $(M+A_i)_{i\in I}$ is ergodic let $g\in G$ and observe 
	$\theta((M+A_i+g)\Delta (M+A_i))\leq \theta(((M+g)+A_i)\setminus A_i)+\theta(((M-g)+A_i)\setminus A_i)\leq \theta(((M+g)+A_i)\Delta A_i)+\theta(((M-g)+A_i)\Delta A_i)$. As $(A_i)_{i\in I}$ is F\o lner we obtain $(M+A_i)_{i\in I}$ to be ergodic from $\lim_{i\in I}{\theta(M+A_i)}/{\theta(A_i)}=1$.		
	To show that $(M+A_i)_{i\in I}$ is Van Hove it is thus  sufficient to show that there exists a compact neighbourhood $U$ of $0$ such that $\lim_{i\in I}\theta(\partial_U (M+A_i))\theta(M+A_i)=0$. Choose a compact and symmetric neighbourhood $U$ of $0$ such that $U+U\subseteq M$. Then for $g\in \partial_U(M+A_i)$ we have $g\in U+M+A_i$ and $g\in U+\overline{(M+A_i)^c}$ and thus there is $u\in U=-U$ such that $g+u\in \overline{(M+A_i)^c}$. As $U$ is a neighbourhood of $0$ there is $v\in U$ such that $g+u+v\in (M+A_i)^c$ and thus $g\notin A_i+M-u-v\supseteq A_i$. This shows $g\in (U+M+A_i)\setminus A_i$ and we obtain 
			\begin{align*}
			\partial_U(M+A_i)
			\subseteq {(U+M+A_i)}\setminus A_i=(U+M+A_i)\Delta A_i.	 	
		\end{align*}
		Thus the statement follows from  $\lim_{i\in I}{\theta(M+A_i)}/{\theta(A_i)}=1$ and $(A_i)_{i\in I}$ being F\o lner. 
	\end{proof}

	\begin{proposition}\label{pro:VanHovenets-K}
		Let $(A_i)_{i\in I}$ be a Van Hove net in an LCA group $G$ and $M\subseteq G$ be a compact subset. Then $(M+A_i)_{i\in I}$ is a Van Hove net and satisfies $\lim_{i\in I}\theta(M+A_i)/\theta(A_i)=1$. Furthermore there exists a Van Hove net $(B_i)_{i\in I}$ with the same index set such that $M+ B_i\subseteq A_i$ for all $i\in I$ and such that $\lim_{i\in I}\theta(B_i)/\theta(A_i)=1$. 
	\end{proposition} 
	\begin{proof}
	See \cite[Proposition 2.3]{04} for the first statement. 
		To show the second statement we assume first that $M$ is symmetric and contains $0$. 
	Set $B_i:=\overline{\{g\in G;\, M+g\subseteq A_i\}}$ for $i\in I$. As $A_i$ is closed, we obtain $M+B_i\subseteq A_i$ for all $i\in I$. 
		
		\textbf{Claim 1:} For all $i\in I$ we have $B_i\subseteq A_i\subseteq B_i\cup \partial_M A_i$.\\
		As $M$ contains $0$ and $A_i$ is closed we obtain $B_i\subseteq A_i$. Now assume that $g\in A_i$. Then $g\in M+\overline{A_i}$. Hence, whenever $g\notin\partial_M A_i=(M+\overline{A_i})\cap (M+\overline{A_i^c})$, then $g\notin M+\overline{A_i^c}\supseteq M+A_i^c$. As $g\notin M+ A_i^c$ is equivalent to  $(M+g)\cap A_i^c=\emptyset$ we obtain $M+g\subseteq A_i$, which implies $g\in B_i$. 
		
		\textbf{Claim 2:} For all $i \in I$ and $K\subseteq G$ compact we have $\partial_K B_i \subseteq \partial_{K+M} A_i$. \\
		From Claim 1 and $0\in M$ we obtain $B_i\subseteq A_i\subseteq M+\overline{A_i}$. Furthermore for $g\in B_i^c$ we have $M+g\not \subseteq A_i$, i.e.\ $(M+g)\cap A_i^c\neq \emptyset$. As $(M+g)\cap A_i^c\neq \emptyset$ is equivalent to $g\in M+A_i^c$ we obtain $B_i^c\subseteq M+A_i^c$. Hence $\overline{B_i^c}\subseteq M+\overline{A_i^c}$ and we compute $\partial_K B_i=(K+B_i)\cap (K+\overline{B_i^c})\subseteq (K+M+A_i)\cap (K+M+\overline{A_i^c})=\partial_{K+M}A_i$. This proves Claim 2. 
		
	Using Claim 1 we compute 
	\[1\geq \frac{\theta(B_i)}{\theta(A_i)}\geq \frac{\theta(A_i)-\theta(\partial_M A_i)}{\theta(A_i)}= 1- \frac{\theta(\partial_M A_i)}{\theta(A_i)}\rightarrow 1.\]	
	Hence $\lim_{i\in I}\theta(B_i)/\theta(A_i)=1$. To show that $(B_i)_{i\in I}$ is a Van Hove net, let $K\subseteq G$ be a compact subset. 
	Using Claim 2 and $\lim_{i\in I}\theta(B_i)/\theta(A_i)=1$ we obtain 
	
\[0\leq \limsup_{i\in I}\frac{\theta(\partial_K B_i)}{\theta(B_i)}\leq \limsup_{i\in I}\frac{\theta(\partial_{K+M}A_i)}{\theta(A_i)}=0, \]
which shows $(B_i)_{i\in I}$ to be a Van Hove net. 
	This shows the statement for all compact and symmetric sets $M$ that contain $0$. 
	The general case can now be achieved as follows. Consider a general compact set $M$. If $M$ is empty, set $B_i:=A_i$. If not, choose $m\in M$. Then $N:=(M-m)\cup -(M-m)$ is compact, symmetric and contains $0$. The arguments above give that there is a Van Hove net $(B_i')_{i\in I}$ such that $N+B_i'\subseteq A_i$ for all $i\in I$ and $\lim_{i\in I}\theta(B_i')/\theta(A_i)=1$. Setting $B_i:=-m+B_i'$, we obtain a Van Hove net $(B_i)_{i\in I}$ that satisfies $\lim_{i\in I}\theta(B_i)/\theta(A_i)=1$ and $M+B_i=M-m+B_i'\subseteq N+B_i'\subseteq A_i$ for all $i\in I$. 
	\end{proof}
	
\begin{remark}
	Note that we have not used the commutativity of the group in the previous proof and thus the statement holds for Van Hove nets in arbitrary locally compact groups. 
\end{remark}

\begin{lemma}
	We have $\lim_{i\in I}\theta(A_i)=\infty$, whenever $(A_i)_{i\in I}$ is a Van Hove net in a non-compact LCA group $G$. 
	\end{lemma}
\begin{proof}
	Let $M\subseteq G$ compact. By Proposition \ref{pro:VanHovenets-K} there exists a Van Hove net $(B_i)_{i\in I}$ such that $B_i+M\subseteq A_i$. As $\theta(B_i)>0$ for large $i$ we obtain in particular $B_i\neq \emptyset$ and thus $\theta(M)\leq \theta(B_i+M)\leq \theta(A_i)$ for large $i$. Hence $\liminf_{i\in I}\theta(A_i)\geq \theta(M)$. As $G$ is assumed to be non-compact we obtain the statement from $\sup_{M\in \mathcal{K}(G)}\theta(M)=\theta(G)=\infty$. 
\end{proof}

\subsubsection{The Ornstein-Weiss lemma}\label{sub:OWG}
	The Ornstein-Weiss lemma is the key tool in order to define entropy  for LCA groups and goes back to	\cite{OrnsteinandWeiss,ward1992abramov, gromov1999topological,lindenstrauss2000mean, Krieger,FeketesLemma,ceccherini2020expansive} in the context of countable discrete amenable groups. A function $f\colon \mathcal{K}(G)\to \mathbb{R}$ is called \emph{subadditive}, if for all $A,B\in \mathcal{K}(G)$ we have
	$f(A\cup B)\leq f(A)+f(B).$ 
Furthermore $f$ is said to be \emph{invariant}, if for all $A\in \mathcal{K}(G)$ and for all $g\in G$ we have
		$f(Ag)= f(A).$ 
Furthermore we call $f$ \emph{monotone}, if for all $A,B\in \mathcal{K}(G)$ with $A\subseteq B$ we have
		$f(A)\leq f(B).$	
From \cite[Corollary 1.3]{04} we obtain the following version of the Ornstein-Weiss lemma.  
\begin{lemma}\label{lem:OrnsteinWeiss}
	Let $G$ be an LCA group that contains a Meyer set. Then for any subadditive, invariant and monotone function $f\colon \mathcal{K}(G)\to \mathbb{R}$ and for every Van Hove net $(A_i)_{i\in I}$ in $G$ the limit 
	\[\lim_{i\in I} \frac{f(A_i)}{\mu(A_i)}\]
	exists, is finite and does not depend on the choice of the Van Hove net. 
\end{lemma}

\subsection{Topological entropy}
Entropy may be one of the most important concepts presenting the complexity of a dynamical system. The notion of topological entropy goes back to \cite{adler1965topological} and was studied also in the context of actions of countable discrete amenable groups \cite{ward1992abramov, OrnsteinandWeiss, weiss2003actions, ollagnier2007ergodic, kerr2016ergodic}, where we are far from giving a full list of the important references. 
Let $\pi$ be an action of an LCA group $G$ on a compact Hausdorff space $X$.
For any entourage $\eta\in \mathbb{U}_X$ and a compact subset $A\subseteq G$ we define the \emph{Bowen entourage} as 
\[\eta_A:=\left \{(x,y);\, \forall g\in {A} : (\pi^g(x),\pi^g(y))\in \eta\right\}=\bigcap_{g\in {A}}\left(\pi^g \times \pi^g\right)^{-1}(\eta).\] 
It is shown in \cite[Lemma 4.2]{04}, that the continuity of $\pi$ and the compactness of $X$ yield that $\eta_A$ is indeed an entourage. We have $\eta_{(A+B)}=(\eta_A)_B$, $\eta_{A\cup B}=\eta_A\cap\eta_B$ and $\eta_A\kappa_A\subseteq (\eta \kappa)_A$ for $\eta,\kappa\in \mathbb{U}_X$ and $A,B\subseteq G$ compact.


For $\eta\in \mathbb{U}_X$ we say that a subset $M\subseteq X$ is \emph{$\eta$-small}, if $M^2\subseteq \eta$. We say, that a set $\mathcal{U}$ of subsets of $X$ is of \emph{scale $\eta$}, if $U$ is $\eta$-small for every $U\in \mathcal{U}$. We denote by 
	$\operatorname{cov}_X(\eta)$ 
the minimum cardinality of an open cover of $X$ of scale $\eta$. This quantity is well defined by the compactness of $X$. 
%
It is presented in \cite[Subsection 4.2]{04} that the function $\mathcal{K}(G)\ni A\mapsto \log (\operatorname{cov}_X(\eta_{A}))$ is monotone, sub-additive and invariant for every $\eta\in \mathbb{U}_X$. If $G$ is assumed to contain Meyer sets, then the limit in the following definition of relative topological entropy exists and is independent of the choice of the Van Hove net $(A_i)_{i\in I}$. For $\eta$ we define the \emph{topological entropy of $\pi$ on scale $\eta$} as
	\[\operatorname{E}(\eta|\pi):=\lim_{i\in I} \frac{\log(\operatorname{cov}_X(\eta_{A_i}))}{\theta(A_i)}.\] 
	We furthermore define the \emph{topological entropy} of $\pi$ as
	$\operatorname{E}(\pi):=\sup_{\eta \in \mathbb{U}_X}\operatorname{E}(\eta,\pi).$
	
\begin{remark}\label{rem:FolnerandTopEntropy1}
	 We will see in Remark \ref{rem:Etopandergodic} below, that the definition of topological entropy depends on the choice of an ergodic net. Nevertheless for every  F\o lner net $(A_i)_{i\in I}$ we have 
	 $\operatorname{E}(\pi)=\sup_{\eta \in \mathbb{U}_X} \limsup_{i\in I} {\log(\operatorname{cov}_X(\eta_{A_i}))}/{\theta(A_i)}$  and a similar formula for the limit inferior. Indeed, for a compact neighbourhood $M$ of $0$ we obtain $(M+A_i)_{i\in I}$ to be a Van Hove net. As $\eta_M\in \mathbb{U}_X$ for all $\eta\in \mathbb{U}_X$ and $\lim_{i\in I}\theta(M+A_i)/\theta(A_i)=1$ we compute 
	\begin{align*}
		\operatorname{E}(\pi)&=\sup_{\eta\in \mathbb{U}_X}\lim_{i\in I}\frac{\log(\operatorname{cov}_X(\eta_{M+A_i}))}{\theta(M+A_i)}= \sup_{\eta\in \mathbb{U}_X}\lim_{i\in I}\frac{\log(\operatorname{cov}_X((\eta_M)_{A_i}))}{\theta(A_i)}\\
	&\leq \sup_{\epsilon\in \mathbb{U}_X}\liminf_{i\in I}\frac{\log(\operatorname{cov}_X(\epsilon_{A_i}))}{\theta(A_i)}
	\leq \sup_{\epsilon\in \mathbb{U}_X}\limsup_{i\in I}\frac{\log(\operatorname{cov}_X(\epsilon_{A_i}))}{\theta(A_i)}\\
	&\leq \sup_{\epsilon\in \mathbb{U}_X}\lim_{i\in I}\frac{\log(\operatorname{cov}_X(\epsilon_{M+A_i}))}{\theta(M+A_i)}=\operatorname{E}(\pi). 
	\end{align*}
\end{remark}

\section{Patch counting for general Delone sets}\label{sec:PatchcountingGeneral}

\subsection{Topological entropy via dense subsets}\label{sec:Entropy}

In Delone dynamical systems one knows the orbit $D_\omega$ to be dense in $X_\omega$. We will thus show that topological entropy can be calculated by considering separating subsets or spanning subsets of a dense subset instead of the whole space. During this section let $\pi$ be an action of an LCA-group that contains a Meyer set on a compact Hausdorff space $X$ and $D\subseteq X$ be a dense subset of $X$. 
A subset $S\subseteq X$ is called \emph{$\eta$-separated}, if for every $s\in S$ there is no further element in $S$ that is $\eta$-close to $s$. 
We say that $S\subseteq X$ is \emph{$\eta$-spanning} for $D$, if for all $d\in D$ there is $s\in S$ such that $s$ is $\eta$-close to $d$ or $d$ is $\eta$-close to $s$.
The following is straightforward. For reference see \cite[Lemma 4.16]{04}.

\begin{lemma} \label{lem:Spasepcov}
	For $\eta\in \mathbb{U}_X$ the cardinality of an  $\eta$-separated subset of $D$ is bounded from above by $\operatorname{cov}_X(\eta)<\infty$ and every $\eta$-separated subset of $D$ of maximal cardinality is $\eta$-spanning for \nolinebreak$D$. 
\end{lemma}

In particular there are finite $\eta$-separated subsets of $D$ of maximal cardinality and finite subsets of $D$ that are $\eta$-spanning for $D$ of minimal cardinality.
We denote by $\operatorname{sep}_D(\eta)$ the maximal cardinality of an $\eta$-separated subset of $D$. Furthermore $\operatorname{spa}_D(\eta)$ is defined as the minimal cardinality of a subset of $D$ that is $\eta$-spanning for $D$.

\begin{lemma}\label{lem:SpaSepCovStaticRelation}
	For $\eta\in \mathbb{U}_X$, there is $\kappa\in \mathbb{U}_X$ such that for all symmetric $\epsilon\in \mathbb{U}_X$ with $\epsilon\subseteq \kappa$ and all $A\subseteq G$ compact we have 
	\begin{itemize}
		\item[(i)] $\operatorname{spa}_{D}(\eta_A)\leq \operatorname{sep}_{D}(\eta_A)\leq \operatorname{cov}_{X}(\eta_A)$ and 
	 	\item[(ii)]   $\operatorname{cov}_{X}(\eta_A)\leq \operatorname{spa}_{D}(\epsilon_A)$.
	 	\end{itemize}
\end{lemma}
\begin{proof}
	From Lemma \ref{lem:Spasepcov} we obtain (i). 
	To show (ii) let $\kappa\in \mathbb{U}_X$ be a symmetric entourage such that $\kappa\kappa\kappa\kappa\subseteq \eta$. Then for $A\subseteq G$ compact and every symmetric entourage $\epsilon\subseteq \kappa$ we have $\epsilon_A\epsilon_A\epsilon_A\epsilon_A\subseteq(\epsilon\epsilon\epsilon\epsilon)_A\subseteq \eta_A$. Let $\chi\in \mathbb{U}_X$ be symmetric and open such that $\chi\subseteq \epsilon_A$. If $S\subseteq D$ is $\epsilon_A$-spanning for $D$, then $\{\chi\epsilon_A[s];\, s\in S\}$ is an open cover of $\epsilon_A[D]=X$ of scale $\chi\epsilon_A\epsilon_A\chi$ and we obtain $\operatorname{cov}_X(\chi\epsilon_A\epsilon_A\chi)\leq\operatorname{spa}_D(\epsilon_A)$. As $\chi\epsilon_A\epsilon_A\chi\subseteq \eta_A$ we obtain the statement. 
	 \end{proof}

\begin{proposition} \label{pro:ENTROPY_SPA_SEP_COV_D}
		Let $\pi$ be an action of an LCA-group that contains Meyer sets on a compact Hausdorff space $X$ and let $D\subseteq X$ be a dense subset. For any F\o lner net $(A_i)_{i\in I}$ and any base $\mathbb{B}_X$ of $\mathbb{U}_X$ we have 
\[
\operatorname{E}(\pi)
=\sup_{\eta \in \mathbb{B}_X} \limsup_{i\in I} \frac{\log(\operatorname{spa}_D(\eta_{A_i}))}{\theta(A_i)}=\sup_{\eta \in \mathbb{B}_X} \limsup_{i\in I} \frac{\log(\operatorname{sep}_D(\eta_{A_i}))}{\theta(A_i)}.\]
Similar statements are satisfied whenever a limit inferior is considered.
\end{proposition} 
\begin{proof} 
By a simple monotonicity argument we can restrict to $\mathbb{B}_X=\mathbb{U}_X$. 
By Lemma \ref{lem:SpaSepCovStaticRelation}(i) we obtain 
\begin{align*}
 \sup_{\eta\in \mathbb{U}_X} \limsup_{i\in I} \frac{\log(\operatorname{spa}_D(\eta_{A_i}))}{\theta(A_i)}
	&\leq \sup_{\eta\in \mathbb{U}_X} \limsup_{i\in I} \frac{\log(\operatorname{sep}_D(\eta_{A_i}))}{\theta(A_i)}\\
	&\leq \sup_{\eta\in \mathbb{U}_X}\limsup_{i\in I} \frac{\log(\operatorname{cov}_X(\eta_{A_i}))}{\theta(A_i)}=\operatorname{E}(\pi)
\end{align*}
 	From Lemma \ref{lem:SpaSepCovStaticRelation}(ii) it follows that
 	\[\operatorname{E}(\pi)=\sup_{\eta\in \mathbb{U}_X}\limsup_{i\in I} \frac{\log(\operatorname{cov}_X(\eta_{A_i}))}{\theta(A_i)}
			\leq \sup_{\epsilon\in \mathbb{U}_X} \limsup_{i\in I} \frac{\log(\operatorname{spa}_D(\epsilon_{A_i}))}{\theta(A_i)}.\] 
An analogue argument shows the result for the limit inferior. 
\end{proof}

\begin{remark}\label{rem:Etopandergodic}
	Note that by the arguments presented in this subsection yield that for every ergodic net $(A_i)_{i\in I}$ we have 
	\[\sup_{\eta \in \mathbb{B}_X} \limsup_{i\in I} \log(\operatorname{cov}_X(\eta_{A_i})) /\theta(A_i) = \sup_{\eta \in \mathbb{B}_X} \limsup_{i\in I} \log(\operatorname{spa}_D(\eta_{A_i}))/\theta(A_i).\]
	and similar formulas concerning $\operatorname{sep}_D$ and the limit inferior. 
	
	Nevertheless $\operatorname{E}(\pi)$ is in general not independent of the choice of an ergodic net. Consider for example the Delone set $\omega:=\mathbb{Z}\setminus \{0\}$ in $\mathbb{R}$ and the corresponding Delone dynamical system $\pi_\omega$. Note that $X_\omega=\{\mathbb{Z}-g;\, g\in \mathbb{R}\}\cup \{\omega-g;\, g\in \mathbb{R}\}$ and that $E(\pi_\omega)=0$. Set $F_n:=\mathbb{Z}\cap [1,e^n]$ and $A_n:=[1,n]\cup F_n$.  As $(A_n+g)\Delta A_n\subseteq \mathbb{Z}\cup (\mathbb{Z}+g)\cup (([1,n]+g)\Delta [1,n])$ for all $g\in\mathbb{R}$ we obtain $(A_n)_{n\in \mathbb{N}}$ to be an ergodic net. Let $V$ be the open and centred ball of radius $1/2$. One easily sees that $M:=\{\omega-f;\, f\in F_n\}$ is an $\epsilon_\omega([-e^n,0],V)$-separated subset of $D_\omega$. Setting $\eta:=\epsilon_\omega\left([0,1],V\right)$ we obtain from Lemma \ref{lem:epsilonentourageandBowen} below that $\eta_{A_n}=\epsilon_\omega\left([0,1]-A_n,V\right)\subseteq \epsilon_\omega([-e^n,0],V)$ and thus $M$ to be $\eta_{A_n}$-separated. Thus by Lemma \ref{lem:SpaSepCovStaticRelation} we get $\operatorname{cov}_{X_\omega}(\eta_{A_n})\geq \operatorname{sep}_{D_\omega}(\eta_{A_n})\geq |M|=|F_n|\geq e^n-1$. We compute 
	$\liminf_{n\to \infty}\log(\operatorname{cov}_{X_\omega}(\eta_{A_n}))/\theta(A_n)\geq 1$. 
\end{remark}

\subsection{Non centred patch counting}

In order to show that the topological entropy of a Delone dynamical system can be calculated via $\operatorname{pat}_\omega$ we introduce intermediate concepts between $A$-patch representations and spanning sets in the corresponding Delone dynamical system. This shows in particular that our definition of patch counting entropy is equivalent to the definition given in \cite{frank2012fusion}.
	During this section we assume $G$ to be an LCA group that contains Meyer sets and $\omega\subseteq G$ to be a Delone set. 
Let $A\subseteq G$ be a compact subset and $V$ be an open neighbourhood of $0$. We say that $F\subseteq G$ is a \emph{non-centred $A$-patch representation} of \emph{scale} $V$ for $\omega$, if for any $g\in G$ there is $f\in F$ s.t. 
\[\omega-g \overset{A,V}{\approx} \omega-f.\]
Note that this is equivalent to $\{\omega-f;\, f\in F\}$ being an $\epsilon_\omega(A,V)$-spanning for $D_\omega$.
Thus by Lemma \ref{lem:Spasepcov} we obtain that there is a finite non-centred  $A$-patch representation on scale $V$ for $\omega$. We define $\operatorname{npat}_\omega(A,V)$ 
	as the minimal cardinality of a non-centred $A$-patch representation of scale $V$ for $\omega$, i.e.\ $\operatorname{npat}_\omega(A,V):=\operatorname{spa}_{D_\omega}(\epsilon_\omega(A,V))$.  
	
	\begin{lemma} \label{lem:epsilonentourageandBowen}
	For $A,K\subseteq G$ compact and any open neighbourhood $V$ of $0$ we have \[\epsilon_\omega(K,V)_A=\epsilon_\omega(K-A,V).\]
\end{lemma}
\begin{proof} To show $\epsilon_\omega(K-A,V)\subseteq \epsilon_\omega(K,V)_A$, let $(\xi,\zeta)\in \epsilon_\omega(K-A,V)$. For $g\in A$ we obtain $\xi\cap (K-g)\subseteq \xi\cap (K-A)\subseteq \zeta+V$, hence 
\[\pi_\omega^g(\xi)\cap K=(\xi+g)\cap K\subseteq \zeta +g+V=\pi_\omega^g(\zeta)+V.\]
Similarly one shows $\pi_\omega^g(\zeta)\cap K\subseteq \pi_\omega^g(\xi)+V$. This proves $(\pi_\omega^g(\xi),\pi_\omega^g(\zeta))\in \epsilon_\omega(K,V)$, i.e.\ $(\xi,\zeta)\in \epsilon_\omega(K,V)_g$ for every $g\in A$. We thus obtain
\[\epsilon_\omega(K-A,V)\subseteq \bigcap_{g\in A}\epsilon_\omega(K,V)_g=\epsilon_\omega(K,V)_A.\]

It remains to show  $\epsilon_\omega(K,V)_A\subseteq \epsilon_\omega(K-A,V)$. For $(\xi,\zeta)\in \epsilon_\omega(K,V)_A$ we have $(\pi_\omega^g(\xi),\pi_\omega^g(\zeta))\in \epsilon_\omega(K,V)$ for every $g\in A$, hence 
\[(\xi+g)\cap K=\pi_\omega^g(\xi)\cap K\subseteq \pi_\omega^g(\zeta)+V=\zeta +g+V.\]
We obtain $\xi\cap(K-g)\subseteq \zeta+V$ for all $g\in A$ and compute 
\[\xi\cap (K-A)=\xi\cap\left(\bigcup_{g\in A}(K-g)\right)=\bigcup_{g\in A}\left(\xi\cap(K-g)\right)\subseteq \zeta+V.\]
As one shows similarly that $\zeta\cap(K-A)\subseteq \xi+V$, we conclude $(\xi,\zeta)\in \epsilon_\omega(K-A,V)$. 
\end{proof}

\begin{proposition}\label{pro:Enpat=Etop}
	For every Van Hove net $(A_i)_{i\in I}$ we have
\[\operatorname{E}(\pi_\omega)=\sup_{V\in \mathcal{N}(G)} \liminf_{i\in I} \frac{\log(\operatorname{npat}_\omega(A_i,V))}{\theta(A_i)}=\sup_{V\in \mathcal{N}(G)} \limsup_{i\in I} \frac{\log(\operatorname{npat}_\omega(A_i,V))}{\theta(A_i)}.\]
\end{proposition}
\begin{proof}
Let $V$ be an open neighbourhood of $0$ and set $\epsilon:=\epsilon_\omega(\{0\},V)$. Then $\operatorname{npat}_\omega(A_i,V)=\operatorname{spa}_{D_\omega}(\epsilon_\omega(A_i,V))=\operatorname{spa}_{D_\omega}(\epsilon_{(-A_i)})$ and we obtain from Proposition \ref{pro:ENTROPY_SPA_SEP_COV_D} that \[\operatorname{E}(\pi_\omega)\geq \limsup_{i\in I}\frac{\log (\operatorname{spa}_{D_\omega}(\epsilon_{(-A_i)}))}{\theta(A_i)}=\limsup_{i\in I}\frac{\log(\operatorname{npat}_\omega(A_i,V))}{\theta(A_i)}.\]
Thus $\operatorname{E}(\pi_\omega)\geq\sup_{W\in \mathcal{N}(G)} \limsup_{i\in I}{\log(\operatorname{npat}_\omega(A_i,W))}/{\theta(A_i)}.$

	Let now $\eta\in \mathbb{U}_X$ and choose $\kappa\in \mathbb{U}_X$ as in Lemma \ref{lem:SpaSepCovStaticRelation}(ii). There are $K\subseteq G$ compact and an open neighbourhood $V$ of $0$ such that $\epsilon_\omega(K,V)\subseteq \kappa$. 
	By Lemma \ref{pro:VanHovenets-K} there exists a Van Hove net $(B_i)_{i\in I}$ such that $\lim_{i\in I}\theta(B_i)/\theta(A_i)=1$ and such that $B_i-K\subseteq -A_i$ for all $i\in I$. Thus we obtain from Lemma \ref{lem:epsilonentourageandBowen} that $\operatorname{cov}_{X_\omega}(\eta_{B_i})\leq \operatorname{spa}_{D_\omega}(\epsilon_\omega(K,V)_{B_i}) =\operatorname{npat}_\omega(K-B_i,V)\leq \operatorname{npat}_\omega(A_i,V)$. Hence
\begin{align*}
	\operatorname{E}(\eta|\pi_\omega)&=\lim_{i\in I}\frac{\log(\operatorname{cov}_{X_\omega}(\eta_{B_i}))}{\theta(B_i)} \\
	&\leq \liminf_{i\in I}\frac{\log(\operatorname{npat}_\omega(A_i,V))}{\theta(A_i)}\\
	&\leq\sup_{W\in \mathcal{N}(G)}\liminf_{i\in I}\frac{\log(\operatorname{npat}_\omega(A_i,W))}{\theta(A_i)}.
\end{align*}
	Taking the supremum over all $\eta\in \mathbb{U}_X$ we obtain 
\[\operatorname{E}(\pi_\omega)\leq\sup_{W\in \mathcal{N}(G)} \liminf_{i\in I} \frac{\log(\operatorname{npat}_\omega(A_i,W))}{\theta(A_i)}\leq\sup_{W\in \mathcal{N}(G)} \limsup_{i\in I} \frac{\log(\operatorname{npat}_\omega(A_i,W))}{\theta(A_i)}\]
and the statement follows. 
\end{proof}

%
%
%
%
%
%
%
%
%
%

\subsection{Centred and non-centred patch counting}\label{sub:centredanduncentred}

We now establish the connection between non-centred and centred $A$-patch representations. Recall from the Introduction that for a Delone set $\omega\subseteq \mathbb{R}^d$, $A\subseteq G$ compact and an open neighbourhood $V$ of $0$ we say that $F\subseteq \omega$ is an \emph{$A$-patch representation} of \emph{scale} $V$ for $\omega$, if for any $g\in \omega$ there is $f\in F$ s.t. \[\omega-f \overset{A,V}{\approx} \omega-g.\]
	Note that an non-centred $A$-patch representation of scale $V$ is not necessarily contained in $\omega$ and thus not necessarily an $A$-patch representation of scale $V$. Nevertheless we have the following. 
	
\begin{lemma} \label{lem:boundonApatchrepresentations}
	Let $V\subseteq G$ be a pre-compact, symmetric and open neighbourhood of $0$ and $A\subseteq G$ be compact. Then for every non-centred $\left(A+\overline{V}\right)$-patch representation $F$ of scale $V$ there exists an $A$-patch representation $E$ of scale $V+V$ such that $|E|\leq|F|$. 
\end{lemma}
\begin{proof}
	For $f\in F$ set $[f]:=\{g\in \omega;\, (\omega-g,\omega-f)\in \epsilon_\omega\left(A+\overline{V},V\right)\}$. For $f\in F$ choose $g_f\in [f]$, whenever $[f]\neq\emptyset$. Otherwise choose $g_f\in G$ arbitrary. Set $E:=\{g_f;\, f\in F\}$. As $F$ is a non- centred $A$-patch representation we obtain that $\bigcup_{f\in F}[f]=\omega$. Thus for $g\in \omega$ there is $f\in F$ with $g\in [f]$ and in particular $[f]\neq \emptyset$. As $g_f\in [f]$ we obtain $(\omega-f)\cap \left(A+\overline{V}\right)\subseteq \omega-g_f+V$. Thus for any $v\in V$ we have $(\omega-f)\cap(A+v)\subseteq \omega-g_f+V$, i.e.
	\[(\omega-f-v)\cap A\subseteq \omega-g_f+V-v\subseteq \omega-g_f+V+V.\]
	 As $V=-V$ we get 
	\[(\omega-f+V)\cap A=\bigcup_{v\in V}(\omega-f-v)\cap A\subseteq \omega-g_f+V+V.\]
	We thus obtain from $g\in [f]$ that
	\[(\omega-g)\cap A\subseteq (\omega-f+V)\cap A\subseteq \omega-g_f+V+V.\]
	Similarly one shows $(\omega-g_f)\cap A\subseteq \omega-g+V+V$. This proves 
	\[\omega-g \overset{A,V+V}{\approx} \omega-g_f\]
	and we obtain $E$ to be an $A$-patch representation of scale $V+V$.
\end{proof}

\begin{remark}\label{rem:repandspafinite}
	For all compact subsets $A\subseteq G$ and all open neighbourhoods $V$ of $0$ there exists a finite $A$-patch representation of scale $V$. Indeed, there is a pre-compact and symmetric open neighbourhood $W$ of $0$ such that $W+W\subseteq V$ and we obtain from Lemma \ref{lem:boundonApatchrepresentations} the existence of a finite $A$-patch representation of scale $W+W$, which is also a finite $A$-patch representation of scale $V$.  
\end{remark}

	We can thus define $\operatorname{pat}_\omega(A,V)$ as the minimal cardinality of an $A$-patch representation of scale $V$ for $\omega$. 

\begin{lemma} \label{lem:centeredornotApatches}
	Let $\omega\subseteq G$ be a Delone set and $K\subseteq G$ compact such that $\omega$ is $K$-dense. Let furthermore $V\subseteq G$ be a pre-compact, symmetric and open neighbourhood of $0$. Then there is a constant $N\in \mathbb{N}$ such that for every compact $A\subseteq G$ we have 
	$\operatorname{pat}_\omega(A,V+V)\leq \operatorname{npat}_\omega\left(A+\overline{V},V\right)$
	and	
	$\operatorname{npat}_\omega(A,V+V)\leq N \operatorname{pat}_\omega(A+K,V).$
	\end{lemma}
	
	\begin{proof} 
	As $K$ is compact, there is a finite set $F_{K,V}\subseteq K$ such that $K\subseteq F_{K,V}+V$. Set $N:=|F_{K,V}|$. The first inequality easily follows from Lemma \ref{lem:boundonApatchrepresentations}. 
	To show the second inequality let $F$ be an $(A+K)$-patch representation of scale $V$ of minimal cardinality. To show that $F+F_{K,V}$ is a non-centred  $A$-patch representation of scale $V+V$ let $g\in G$. 
	As $K+\omega=G$, there are $e\in F_{K,V}$, $v\in V$ and $u\in \omega$ with $e+v\in K$ and $e+v+u=g$. 
Since $F$ is an $A+K$-patch representation of scale $V$ there is $f\in F$ with 
	 \[\omega-f\overset{A+K,V}{\approx}\omega-u.\] 
	From $e\in F_{K,V}\subseteq K$ we get 
	$(\omega-f)\cap (A+e)\subseteq (\omega-u)+V.$
	We thus compute
	\begin{align*}
		( \omega-(f+e) ) \cap A&= ((\omega-f)\cap (A+e))-e\\
		&\subseteq (\omega-u) +V -e\\
		&=\omega -g+v+V\\
		&\subseteq (\omega -g) +(V+V).
	\end{align*}	
	As $e+v\in K$ we obtain that 
	$(\omega-u)\cap (A+e+v)\subseteq (\omega-f)+V.$
	Thus $V=-V$ implies
	\begin{align*}
		(\omega-g)\cap A &=(\omega-e-v-u)\cap A\\
			&=((\omega-u)\cap(A+e+v))-e-v\\
			&\subseteq (\omega-f)+V-e-v\\
			&\subseteq (\omega-(f+e))+(V+V).
	\end{align*}
	Hence $\omega-g \overset{A,V+V}{\approx} \omega-(f+e)$. As $f+e\in F+F_{K,V}$ we have shown that $F+F_{K,V}$ to be a non-centred  $A$-patch representation of scale $V+V$. We thus obtain 
	$\operatorname{npat}_\omega(A,V+V)\leq |F+F_{K,V}|\leq N|F|=N\operatorname{pat}_\omega(A+K,V).$
	\end{proof}

\begin{theorem}\label{the:Etop=Epat}
	Assume that $G$ is a non-compact LCA group that contains Meyer sets. For every Delone set $\omega\subseteq G$ and every Van Hove net $(A_i)_{i\in I}$ we have  
	\[\operatorname{E}(\pi_\omega)=\sup_{V\in \mathcal{N}(G)}\liminf_{i\in I}\frac{\log(\operatorname{pat}_\omega(A_i,V))}{\theta(A_i)}=\sup_{V\in \mathcal{N}(G)}\limsup_{i\in I}\frac{\log(\operatorname{pat}_\omega(A_i,V))}{\theta(A_i)}.\] 
\end{theorem}
\begin{proof}
	Let $K\subseteq G$ be compact such that $\omega$ is $K$ discrete and choose a Van Hove net $(B_i)_{i\in I}$ such that $\lim_{i\in I}\theta(B_i)/\theta(A_i)=1$ and such that $B_i+K\subseteq A_i$ for all $i\in I$. 
	Let $V\subseteq G$ be an open neighbourhood of $0$. Then there exists a pre-compact and open neighbourhood $W$ of $0$ such that $W+W\subseteq V$. By Lemma \ref{lem:centeredornotApatches} there is $N\in \mathbb{N}$ such that for all $i\in I$ we have
	$\operatorname{npat}_\omega(B_i,V)\leq \operatorname{npat}(B_i,W+W)\leq N\operatorname{pat}_\omega(B_i+K,W)\leq N \operatorname{pat}_\omega(A_i,W)$ and we obtain from Proposition \ref{pro:Enpat=Etop} that 
	\begin{align*}
	\operatorname{E}(\pi_\omega)
		&=\liminf_{i\in I}\frac{\log(\operatorname{npat}_\omega(B_i,V))}{\theta(B_i)}
		\leq \liminf_{i\in I}\frac{\log(N)+\log(\operatorname{pat}_\omega(A_i,W))}{\theta(A_i)}\\
		&=\liminf_{i\in I} \frac{\log(\operatorname{pat}_\omega(A_i,W))}{\theta(A_i)}
		\leq \sup_{U\in \mathcal{N}(G)}\liminf_{i\in I} \frac{\log(\operatorname{pat}_\omega(A_i,U))}{\theta(A_i)}.
	\end{align*}
The proof of the reverse inequality, with a limit superior instead of a limit inferior, is similar: use the other statement from Lemma \ref{lem:centeredornotApatches}. 
\end{proof}


\section{Patch counting for FLC Delone sets}\label{sec:FLC} 

In this section we study the classical definition of patch counting entropy presented in the Introduction. 
Therefore we assume $G$ to be a non-compact LCA group that contains Meyer sets and $\omega\subseteq G$ to be a FLC Delone set. For $A\subseteq G$ compact we call a subset $F\subseteq \omega$ an \emph{exact $A$-patch representation}, if for all $g\in \omega$ there is $f\in F$ such that $(\omega-g)\cap A=(\omega-f)\cap A$. The minimal cardinality of an exact $A$-patch representation is $|\operatorname{Pat}_\omega(A)|$. For every  open neighbourhood $V$ of $0$ we obtain that every exact $A$-patch representation is an $A$-patch representation of scale $V$. Thus for every compact $A\subseteq G$ we have
$\operatorname{pat}_\omega(A,V)\leq |\operatorname{Pat}_\omega(A)|$ and we obtain from Theorem \ref{the:Etop=Epat} the following. 
\begin{proposition}\label{pro:EpatleqEPAT}
	For every Van Hove net $(A_i)_{i\in I}$ we have
	\[\operatorname{E}(\pi_\omega)\leq \liminf_{i\in I}\frac{\log(|\operatorname{Pat}_\omega(A_i)|)}{\theta(A_i)}\leq\limsup_{i\in I}\frac{\log(|\operatorname{Pat}_\omega(A_i)|)}{\theta(A_i)}.\]
\end{proposition}
	
To obtain an equality in the previous proposition we have to restrict to Van Hove nets that are compactly connected to $0$. 
	Let $C\subseteq G$ be a compact subset. We say that $A\subseteq G$ is $C-$connected to $0$, if for all $a\in A$ there are $a_0,\cdots,a_n\in A\cup\{0\}$  with $a_0=0$, $a_n=a$ and $a_i-a_{i-1}\in C$ for every $i\in \{1,\cdots,n\}$. Furthermore we say that a net of compact sets $(A_i)_{i\in I}$ is \emph{$C$-connected to $0$}, if $A_i$ is $C$-connected to $0$ for all $i\in I$. A net is called \emph{compactly connected to $0$}, if it is $C$-connected to $0$ for some compact set $C\subseteq G$. Examples of $C$-connected sets in $\mathbb{R}^d$ can be found in the Introduction. 
	We will use the local matching base $\mathbb{B}_{\text{lm}}(\omega)$ of $\mathbb{U}_{X_\omega}$ in order to establish the equality in Proposition \ref{pro:EpatleqEPAT} for Van Hove nets that are compactly connected to $0$. Unfortunately the formula in Lemma \ref{lem:epsilonentourageandBowen}, which gives the tool to calculate the Bowen entourages of members of the local rubber base $\mathbb{B}_{lr}(\omega)$ does not necessarily hold for members of the local matching base $\mathbb{B}_{lm}(\omega)$. Nevertheless, it is straigh forward to show the following. 
	
	\begin{lemma} \label{lem:Bowenentourageforlmbase}
	For all compact $K\subseteq G$ and every open neighbourhood $V$ of $0$, and $g\in G$ we have $\eta_\omega(K,V)_g=\eta_\omega(K-g,V).$
\end{lemma}

	Using this we can now prove the following key lemma. 

\begin{lemma}\label{lem:PatleqSEP}
Let $C$ be a symmetric and compact neighbourhood of $0$ and $V$ an open neighbourhood of $0$ contained in $C$. Assume $\omega$ to be a $C$-dense and $V$-discrete Delone set. Then there is $\eta\in \mathbb{U}_X$ such that for all compact $A\subseteq G$ that are $C$-connected to $0$ and also contain $0$ we have
\[|\operatorname{Pat}_\omega(A)|\leq \operatorname{sep}_{D_\omega}(\eta_{(-A)}).\]
\end{lemma}
\begin{proof}
Set $M:=C+C+C$, $K:=M+M+M$ and $\eta:=\eta_\omega(K,V)$. 
To show the statement it is sufficient to show that every exact $A$-patch representation $F\subseteq \omega$ of minimal cardinality $|\operatorname{Pat}_\omega(A)|$ satisfies that $\{\omega-g;\, g\in F\}$ is $\eta_{(-A)}$-separated. To argue by contraposition assume $F\subseteq \omega$ to be an exact $A$-patch representation $F\subseteq \omega$ for which $\{\omega-g;\, g\in F\}$ is not $\eta_{(-A)}$-separated. Thus there are $x,y\in F$ such that $(\omega-x,\omega-y)\in \eta_{(-A)}$. We will argue below that $(\omega-x)\cap A=(\omega-y)\cap A$ and thus obtain that $F$ is not a minimal exact $A$-patch representation. 

It remains to show that $(\omega-x)\cap A=(\omega-y)\cap A$. Assume $a\in (\omega-x)\cap A$. As $A$ is $C$-connected there are $a_0,\cdots,a_n\in A$ such that $a_0=0,a_n=a$ and $a_{i+1}-a_i\in C$ for all $i=1,\cdots,n-1$. Set $x_0:=0$ and $x_n:=a$. 
As $\omega-x$ is $C$-dense there are $x_i\in \omega-x$ such that $a_i-x_i\in C$ for $i=1,\cdots, n-1$. Note that $x_0=0\in (\omega-y)\cap A$. We will show $a=x_n\in (\omega-y)$ by induction and thus obtain that $(\omega-x)\cap A\subseteq (\omega-y)\cap A$. Similarly one obtains the other inclusion. 

Assume now $x_i\in \omega-y$ for some $i=0,\cdots,n-1$. As $a_i\in A$ we obtain $(\omega-x,\omega-y)\in \eta_{-A}\subseteq \eta_{-a_i}=\eta_\omega(K+a_i, V)$. Thus there are $u,v\in V$ such that $(\omega-x+u)\cap (K+a_i)=(\omega-y+v)\cap (K+a_i)$ and we obtain $x_i+u=(x_i-a_i)+a_i+u\in C+a_i+V\subseteq K +a_i$. From $x_i+u\in (\omega-x+u)\cap (K+a_i)\subseteq \omega-y+v$ and the inductive assumption we obtain that $x_i,x_i+(u-v)\in \omega-y$. As $\omega-y$ is $V$-discrete we get $x_i=x_i+u-v$, i.e.\ $u=v$ and thus we have 
	$(\omega-x+u)\cap (K+a_i)=(\omega-y+u)\cap (K+a_i)$. 
	It follows that $(\omega-x)\cap (K+a_i+u)=(\omega-y)\cap (K+a_i+u)$. 
	From $M+x_i=M+(x_i-a_i)-u+(a_i+u)\subseteq M+C-V+(a_i+u)\subseteq K+(a_i+u)$ we obtain that $(\omega-x)\cap (M+x_i)=(\omega-y)\cap (M+x_i)$. Thus $x_{i+1}-x_i\in C+C+C=M$ implies $x_{i+1}\in (\omega-x)\cap (M+x_i)\subseteq \omega-y$. 
\end{proof}

\begin{proposition}\label{pro:EPATleqEtop}
	For all compact subsets $C\subseteq G$ there exists an entourage $\eta\in \mathbb{U}_X$ such that for every Van Hove net $(A_i)_{i\in I}$ that is $C$-connected to $0$ and satisfies $0\in A_i$ for all $i\in I$ we have
	\[\limsup_{i\in I}\frac{\log(|\operatorname{Pat}_\omega(A_i)|)}{\theta(A_i)} \leq \operatorname{E}(\eta|\pi_\omega)\leq \operatorname{E}(\pi_\omega).\]
\end{proposition}
\begin{proof}
Let $K\subseteq G$ be compact such that $\omega$ is $K$-dense. Choose a compact and symmetric $\tilde{C}\subseteq G$ that contains $K\cup C$. Choose furthermore an open neighbourhood $V$ of $0$ that is contained in $\tilde{C}$ such that $\omega$ is $V$-discrete and choose $\eta$ according to Lemma \ref{lem:PatleqSEP} with respect to $\tilde{C}$ and $V$. Then by Lemma \ref{lem:SpaSepCovStaticRelation}(i) we obtain that $|\operatorname{Pat}(A_i)|\leq \operatorname{sep}_{D_\omega}(\eta_{(-A_i)})\leq \operatorname{cov}_{X_\omega}(\eta_{(-A_i)})$ and a straightforward computation shows the statement. 
\end{proof}

\begin{lemma}\label{lem:0andconnected}
Let $(A_i)_{i\in I}$ be a Van Hove net and set $B_i:=A_i\cup \{0\}$ for all $i\in I$. Then $(B_i)_{i\in I}$ is a Van Hove net and we have 
\[\limsup_{i\in I}\frac{\log(|\operatorname{Pat}_\omega(A_i)|)}{\theta(A_i)}=\limsup_{i\in I}\frac{\log(|\operatorname{Pat}_\omega(B_i)|)}{\theta(B_i)}.\]
\end{lemma}
\begin{proof}
	Let $K\subseteq G$ be compact. From 
	\begin{align*}
	\partial_K B_i&=(K+(A_i\cup \{0\}))\cap \overline{K+(A_i\cup\{0\})^c}\\
	&\subseteq ((K+A_i)\cup K)\cap (K+\overline{A_i^c})\subseteq (\partial_K A_i)\cup K
	\end{align*}
	we obtain that 
	$0\leq \theta(\partial_K B_i)/\theta(B_i)\leq (\theta(\partial_K A_i)+\theta(K))/\theta(A_i)\rightarrow_{i\in I} 0.$
	Thus $(B_i)_{i\in I}$ is a Van Hove net. Furthermore $1\leq \theta(B_i)/\theta(A_i)\leq (\theta(A_i)+\theta(\{0\}))/\theta(A_i)\rightarrow 1+0$ implies $\lim_{i\in I}\theta(B_i)/\theta(A_i)=1$. 
	As $\operatorname{Pat}_\omega(B_i)\subseteq \operatorname{Pat}_\omega(A_i)\cup \{P\cup \{0\};\, P\in \operatorname{Pat}_\omega(A_i)\}$ for all $i\in I$ we obtain 
	$|\operatorname{Pat}_\omega(A_i)|\leq |\operatorname{Pat}_\omega(B_i)|\leq 2 |\operatorname{Pat}_\omega(A_i)|$
	and a straightforward argument yields the statement. 
	\end{proof}

\begin{theorem}\label{the:EPat=Etop}
	Let $G$ be a non-compact LCA group that contains Meyer sets and $\omega$ be a FLC Delone set in $G$. For every 
	 Van Hove net $(A_i)_{i\in I}$ that is compactly connected to $0$ we have
	\[\operatorname{E}(\pi_\omega)=\lim_{i\in I}\frac{\log(|\operatorname{Pat}_\omega(A_i)|)}{\theta(A_i)}.\]
\end{theorem}
\begin{proof} 
Let $C\subseteq G$ compact such that $A_i$ is $C$-connected to $0$ for all $i\in I$, and define $B_i:=A_i\cup\{0\}$. Then by Lemma \ref{lem:0andconnected} $(B_i)_{i\in I}$ is a Van Hove net that is $C$-connetced to $0$ with $0\in B_i$ and we obtain from Proposition \ref{pro:EPATleqEtop} that
\[\limsup_{i\in I}\frac{\log(|\operatorname{Pat}_\omega(A_i)|)}{\theta(A_i)}=\limsup_{i\in I}\frac{\log(|\operatorname{Pat}_\omega(B_i)|)}{\theta(B_i)}\leq \operatorname{E}(\pi_\omega).\]
	Thus the statement follows from Proposition \ref{pro:EpatleqEPAT}. 
\end{proof}

\begin{corollary}\label{cor:FolnerleqEPat}
	If $(A_i)_{i\in I}$ is a F\o lner net that is compactly connected to $0$, then 
	 \begin{align*}
		\limsup_{i\in I}\frac{\log(|\operatorname{Pat}_\omega(A_i)|)}{\theta(A_i)}\leq \operatorname{E}(\pi_\omega).
	\end{align*}
\end{corollary}
\begin{proof}
	Let $K\subseteq G$ be a compact neighbourhood of $0$. Then $(A_i+K)_{i\in I}$ is a Van Hove net that is compactly connected to $0$ and we compute
	 \begin{align*}
		\limsup_{i\in I}\frac{\log(|\operatorname{Pat}_\omega(A_i)|)}{\theta(A_i)}
		&\leq \limsup_{i\in I}\frac{\log(|\operatorname{Pat}_\omega(A_i+K)|)}{\theta(A_i)}\\
		&=\limsup_{i\in I}\frac{\log(|\operatorname{Pat}_\omega(A_i+K)|)}{\theta(A_i+K)}=\operatorname{E}(\pi_\omega).
	\end{align*} 
\end{proof}

\begin{remark}
	In \cite{baake2007pure} it is shown for FLC Delone sets $\omega\subseteq \mathbb{R}^d$ that 
	\[\operatorname{E}(\pi_\omega)=\limsup_{n \to \infty}\frac{\log(|\operatorname{Pat}_\omega({B}_n)|)}{\theta({B}_n)},\]
	where $({B}_n)_{n\in \mathbb{N}}$ is the Van Hove sequence of the centred closed balls of radius $n\in \mathbb{N}$. In this context the question was raised, whether there is an  overall factor of ${\theta(B_1)}/{\theta(C_1)}$, if we replace the centred balls by centred cubes $C_n$ of side length $2n$. We obtain from Theorem \ref{the:EPat=Etop} that this factor is $1$.
	\end{remark}

From Proposition \ref{pro:EPATleqEtop} we can furthermore conclude the following. 

\begin{theorem}\label{the:EpatsimplifiedforFLC}
	Let $\omega$ be a FLC Delone set in a non-compact LCA group that contains Meyer sets. Let $C\subseteq G$ be a compact subset. Then there exists an open neighbourhood $V$ of $0$ such that for every $C$-connected Van Hove net $(A_i)_{i\in I}$ with $0\in A_i$ for all $i\in I$ and all open neighbourhoods $W$ of $0$ that are contained in $V$ the following limit exists and satisfies
\[\operatorname{E}(\pi_\omega)=\lim_{i\in I}\frac{\log(\operatorname{pat}_\omega(A_i,W))}{\theta(A_i)}.\]
\end{theorem}
\begin{proof} By Proposition \ref{pro:EPATleqEtop} there exists $\eta\in \mathbb{U}_X$ such that $\limsup_{i\in I}\log(|\operatorname{Pat}(A_i)|)/\theta(A_i)\leq \operatorname{E}(\eta|\pi_\omega)$. 
By Lemma \ref{lem:SpaSepCovStaticRelation}(ii) there is $\kappa\in \mathbb{U}_X$ such that for all symmetric $\epsilon\subseteq \kappa$ we have $\operatorname{cov}_{X}(\eta_A)\leq \operatorname{spa}_{D}(\epsilon_A)$ for all compact sets $A\subseteq G$. Let $K'\subseteq G$ be a compact subset and $V'$ be an open neighbourhood of $0$ such that $\epsilon_\omega(K',V')\subseteq \kappa$. Let $K\subseteq G$ be compact such that $\omega$ is $K$-dense and $V$ a pre-compact, symmetric and open neighbourhood of $0$ such that $V+V\subseteq V'$. Let $(B_i)_{i\in I}$ be a Van Hove net such that $B_i+K'+K\subseteq A_i$ for all $i\in I$ and such that $\lim_{i\in I}\theta(B_i)/\theta(A_i)=1$. We thus obtain from Lemma \ref{lem:centeredornotApatches} the existence of a $N\in \mathbb{N}$ such that we have
\begin{align*}
	\operatorname{cov}_{X_\omega}(\eta_{(-B_i)})&\leq \operatorname{spa}_{D_\omega}(\epsilon_\omega(K',V')_{(-B_i)})
	=\operatorname{npat}_\omega(B_i+K',V')
	\leq \operatorname{npat}_\omega(B_i+K',V+V)\\
	&\leq N\operatorname{pat}_\omega(B_i+K'+K,V)
	\leq N \operatorname{pat}_\omega(A_i,V). 
\end{align*}
Thus  by Proposition \ref{pro:EpatleqEPAT} and Theorem \ref{the:Etop=Epat} for every open neighbourhood $W$ of $0$ that is contained in $V$ we observe
\begin{align*}
\operatorname{E}(\pi_\omega)&\leq \limsup_{i\in I}\frac{\log(|\operatorname{Pat}_\omega(A_i)|)}{\theta(A_i)}
\leq \operatorname{E}(\eta|\pi_\omega)
=\lim_{i\in I}\frac{\log(\operatorname{cov}_{X_\omega}(\eta_{(-B_i)}))}{\theta(-B_i)}\\
&\leq \liminf_{i\in I}\frac{\log(N\operatorname{pat}_{\omega}(A_i,W))}{\theta(A_i)}
=\liminf_{i\in I}\frac{\log(\operatorname{pat}_{\omega}(A_i,W))}{\theta(A_i)}\\
&\leq\limsup_{i\in I}\frac{\log(\operatorname{pat}_{\omega}(A_i,W))}{\theta(A_i)}\leq \operatorname{E}(\pi_\omega).
\end{align*}
\end{proof}

\section{Examples}\label{sec:examples}

We now give details on Example \ref{exa:introcompactlyconnectednecessary}. 
\begin{example}\label{exa:compactlyconnectednecessary}
	Consider the finite local complexity Delone set $\omega:=(-\mathbb{N}_0)\cup \alpha\mathbb{N}_0\subseteq \mathbb{R}$ for $\alpha\in [0,1]$ irrational. 	
	Then for $\kappa\in [0,\infty]$ one obtains
	\begin{align*}
		\lim_{n \to \infty}\frac{\log(|\operatorname{Pat}_\omega(A_n)|)}{\theta(A_n)}=\kappa,
	\end{align*}
	whenever we choose $A_n:=[0,n]+e^{\kappa n}$ if $\kappa$ is finite and $A_n:=[0,n]+e^{(n^2)}$ if $\kappa=\infty$. 
	
	We first consider the case $\kappa<\infty$ and define $A_n:=[0,n]+e^{\kappa n}$ for any $n\in \mathbb{N}$. We will now show that $F_n:=\mathbb{Z}\cap[-(n+1)-e^{\kappa n},0]$ is an exact $A_n$-patch representation for $\omega$. To do this we consider $g\in \omega \setminus F_n$. If $g> 0$, then $(\omega-g)\cap A_n=(\alpha\mathbb{Z})\cap A_n=(\omega-0)\cap A_n$ and we can represent $g$ by $0\in F_n$. 
	If $g\leq 0$, then $g\in \mathbb{Z}$ and we obtain from $g\notin F_n$ that we have $g<-(n+1)-e^{\kappa n}$. Thus we obtain $g,(\min F_n)\leq -n-e^{\kappa n}$ and observe $(\omega-g)\cap A_n=\mathbb{Z}\cap A_n=(\omega-\min F_n)\cap A_n$. 
	This shows that $F_n$ is indeed an exact $A_n$-patch representation for $\omega$ and we obtain $|\operatorname{Pat}_\omega(A_n)|\leq |F_n|\leq (n+1)+e^{\kappa n}$. Now as $\kappa\geq 0$ for sufficiently large $n$ we have $1 \leq e^{\kappa n}$ and we get $\log|\operatorname{Pat}_\omega(A_n)| \leq \log((n+2)e^{\kappa n})=\log(n+2)+\kappa n$. Hence 
\begin{align*}
		\limsup_{n \to \infty}
		\frac{\log|\operatorname{Pat}_\omega(A_n)|}{\theta(A_n)}
		\leq \limsup_{n \to \infty} \left(\frac{\log(n+2)}{n}+\frac{\kappa n}{n}\right)=\kappa.
	\end{align*}
Thus the statement follows whenever $\kappa=0$. Otherwise let us next consider $E_n:=\mathbb{Z}\cap(-e^{\kappa n},0]\cap \mathbb{Z}$. Then for $g\in E_n$ we have $|g|< e^{\kappa n}=\min A_n$. Thus the elements of $(\omega-g)\cap A_n$ are of the form $|g|+k\alpha$ for $k\in \mathbb{N}$. Furthermore as $\alpha\leq 1$ there is at least one such number contained in $(\omega-g)\cap A_n$.
	Thus whenever we consider distinct $g,g'\in E_n$ we obtain from $\alpha$ being irrational, that the corresponding patches $(\omega-g)\cap A_n$ and $(\omega-g')\cap A_n$ do not agree. This yields that $|\operatorname{Pat}_\omega(A_n)|\geq |E_n|\geq e^{\kappa n}-1.$ Now as we assume that $\kappa>0$ we obtain that for large $n$ we have $2\leq e^{\kappa n}$ and in particular that $\log|\operatorname{Pat}_\omega(A_n)|\geq \log(e^{\kappa n}-(1/2)e^{\kappa n})=\log(1/2)+\kappa n$. This allows to compute 
 \begin{align*}
		\liminf_{n \to \infty}
		\frac{\log|\operatorname{Pat}_\omega(A_n)|}{\theta(A_n)}
		\geq \liminf_{n \to \infty}
		\frac{\log(1/2)+\kappa n}{n}=\kappa.
	\end{align*}
and we obtain the claimed statement for all $\kappa<\infty$. Similarly one shows the result for $\kappa=\infty$ using $A_n:=[0,n]+e^{(n^2)}$. 
\end{example}

	We next give details on Example \ref{exa:introFolnerrightergodic} in Example \ref{exa:Folnerfails} and Example \ref{exa:ergodicfails}.

\begin{example}\label{exa:Folnerfails}
Consider the Delone set of finite local complexity
			\[\omega:=\{n\in \mathbb{N};\, \xi_n=1\}\cup (\mathbb{Z}+1/2),\] where $(\xi_n)_{n\in \mathbb{N}}$ is a sequence containing all finite words in $\{0,1\}$, i.e.\ for all finite sequences $(x_j)_{j=1}^n$ there exists $i\in \mathbb{N}$ such that $\xi_{i+j}=x_j$ for $j=1,\cdots,n$. Then $\operatorname{E}_{pc}(\omega)=\log(2)$ and for all $\kappa\in[0,\log(2)]$ there is a F\o lner sequence $(A_n)_{n\in \mathbb{N}}$, which is compactly connected to $0$, such that
		\begin{align}\label{form:EPATforFolner2}
		\limsup_{n \to \infty}\frac{\log(\operatorname{Pat}_\omega(A_n))}{\theta(A_n)}=\kappa.
	\end{align}
	Indeed, consider $A_n:=[0,\rho n]\cup \left([0,n]\setminus\left(\frac{1}{2}\mathbb{Z}+B^o_{(n+3)^{-1}}\right)\right)$ with $\rho:=\kappa/\log(2)\in [0,1]$, where $B_r^o$ denotes the open and centred Euclidean ball with radius $r$. 
	We first show that $(A_n)_{n\in \mathbb{N}}$ is a F\o lner sequence. Let $K\subseteq G$ compact and non-empty let $k\in \mathbb{N}$ such that $K\subseteq [-k,k]$. From $\theta(A_n\cap (K+A_n))+\theta(A_n\setminus (K+A_n))=\theta(A_n)\leq \theta(K+A_n)=\theta((K+A_n)\setminus A_n)+\theta((K+A_n)\cap A_n)$ it follows that $\theta((K+A_n)\Delta A_n)\leq 2 \theta((K+A_n)\setminus A_n)\leq 2 \theta(([-k,k]+A_n)\setminus A_n)$. As $([-k,k]+A_n)\setminus A_n\subseteq [-k,n+k]\setminus (\mathbb{R}\setminus(\frac{1}{2}\mathbb{Z}+B^o_{(n+2)^{-1}}))=[-k,n+k]\cap (\frac{1}{2}\mathbb{Z}+B^o_{(n+2)^{-1}})$ we get $\theta((K+A_n)\Delta A_n)/\theta(A_n)\leq 8(2k+n)/( n (n+2))$ and obtain $(A_n)_{n\in \mathbb{N}}$ to be a F\o lner net.
	We next show (\ref{form:EPATforFolner2}). As $\omega\subseteq 1/2 \mathbb{Z}$ for $n\in \mathbb{N}$ we get 
\begin{align*}
\operatorname{Pat}_\omega(A_n)=\operatorname{Pat}_\omega([0,\rho n])
=&\left\{W\cup \left([0,\rho n]\cap \left(\frac{1}{2}+\mathbb{Z}\right)\right);\, W\subseteq [0,\rho n]\cap \mathbb{Z}\right\}\\
&\cup 
	\left\{\left(W-\frac{1}{2}\right)\cup\left([0,\rho n]\cap \mathbb{Z}\right);\, W\subseteq [1,\rho n]\cap \mathbb{Z}\right\}.
\end{align*}
We thus obtain that $2^{\rho n}\leq |\operatorname{Pat}_\omega(A_n)|\leq 2^{\rho n +2}$ and a straightforward argument shows (\ref{form:EPATforFolner2}). 	
	Note that for $\kappa=\log(2)$ we obtain that $(A_n)_{n\in \mathbb{N}}=([0,n])_{n\in \mathbb{N}}$ is a Van Hove sequence that is compactly connected to $0$. Thus $\operatorname{E}_{pc}(\omega)=\log(2)$. 
\end{example}

\begin{example}\label{exa:ergodicfails}	
	Consider $\omega\subseteq \mathbb{R}$ as in Example \ref{exa:Folnerfails}.  		
	Then for all $\kappa\in[0,\infty]$ there is an ergodic sequence $(A_n)_{n\in \mathbb{N}}$, which is compactly connected to $0$, such that 
		\begin{align}\label{form:EPATforergodic2}
		\limsup_{n \to \infty}\frac{\log(\operatorname{Pat}_\omega(A_n))}{\theta(A_n)}=\kappa.
	\end{align}
	As every F\o lner sequence is ergodic it remains to consider $\kappa\in [\log(2),\infty]$. 
	We first consider the case $\kappa<\infty$, set $\rho:=\kappa/\log(2)$ and define  
	\[A_n:= \left([0,n]\setminus \left(\frac{1}{2}\mathbb{Z}+B^o_{(n+3)^{-1}}\right)\right)\cup \big([0,\rho n]\cap \mathbb{Z}\big).\] 
	To show that $(A_n)_{n\in\mathbb{N}}$ is ergodic let $g\in \mathbb{R}$. Then 
	a straightforward computation shows $(A_n+g)\Delta A_n\subseteq \mathbb{Z}\cup (\mathbb{Z}+g)\cup(([0,n]+g)\Delta[0,n])$ and we obtain $\theta((A_n+g)\Delta A_n)/\theta(A_n)\leq \theta(([0,n]+g)\Delta[0,n])/\theta([0,n])\rightarrow 0.$
	This shows $(A_n)_{n\in \mathbb{N}}$ to be ergodic. As $\omega-g\subseteq 1/2 \mathbb{Z}$ for all $g\in \omega$ we obtain $|\operatorname{Pat}_\omega(A_n)|
			=\left|\operatorname{Pat}_\omega([0,\rho n]\cap \mathbb{Z})\right|$. 
	Hence $2^{\rho n}\leq |\operatorname{Pat}_\omega(A_n)|\leq 2^{\rho n +1}$ and we obtain (\ref{form:EPATforergodic2}). Similar one shows (\ref{form:EPATforergodic2}) in the case of $\kappa=\infty$ using the ergodic sequence $A_n:= ([0,n]\setminus (1/2\mathbb{Z}))\cup ([0,n^2]\cap \mathbb{Z})$. 
	\end{example}

\begin{acknowledgement}
The author would like to thank Gabriel Fuhrmann for enlightening discussions on F\o lner and Van Hove sequences and Tobias Oertel-J\"ager for his patient supervision. Furthermore the author would like to express his gratitude towards the anonymous referee for very useful suggestions that improved the readability of the article. 
\end{acknowledgement}



 \vspace{10mm} \noindent
\begin{tabular}{l l }
Till Hauser \\
Faculty of Mathematics and Computer Science\\
Institute of Mathematics \\
Friedrich Schiller University Jena\\
07743 Jena\\
Germany\\
\end{tabular}

\end{document}